\RequirePackage{ifpdf}
\ifpdf
\documentclass[a4paper,11pt,twoside]{article}
\usepackage{microtype}
\else
\documentclass[a4paper,11pt,twoside,dvips]{article}
\fi

\usepackage[scale=0.85,includeheadfoot]{geometry}
\usepackage[medium]{titlesec}

\usepackage[utf8x]{inputenc}
\usepackage[english]{babel}

\usepackage[T1]{fontenc}

\usepackage{amsmath,amsfonts,amsthm, mathtools}
\usepackage{subeqnarray}
\usepackage{latexsym}

\usepackage{mathptmx}
\usepackage{authblk}

\usepackage{parskip}
\usepackage{fancyhdr}
\usepackage{graphicx}
\usepackage{hyperref}

\usepackage{tikz}

\graphicspath{{pics/}}

\usepackage[vlined,lined,algonl,algo2e,boxed]{algorithm2e}

\usepackage{fancyhdr}
\numberwithin{equation}{section}
\numberwithin{figure}{section}
\numberwithin{table}{section}

\definecolor{MyRed}{rgb}{0.6, 0,0}
\definecolor{MyGreen}{rgb}{0.0, 0.5, 0.0}
\definecolor{MyBlue}{rgb}{0.0,0.0,0.5}
\definecolor{MyGray}{rgb}{0.9,0.9,0.9}

\newcommand{\captionfonts}{\small}
\makeatletter  
\long\def\@makecaption#1#2{%
  \vskip\abovecaptionskip
  \sbox\@tempboxa{{\captionfonts #1: #2}}%
  \ifdim \wd\@tempboxa >\hsize
    {\captionfonts #1: #2\par}
  \else
    \hbox to\hsize{\hfil\box\@tempboxa\hfil}%
  \fi
  \vskip\belowcaptionskip}
\makeatother   

\newtheorem{proposition}{Proposition}
 
\newtheorem{theorem}[proposition]{Theorem}
\newtheorem{corollary}[proposition]{Corollary} 
\newtheorem{lemma}[proposition]{Lemma}

\newcommand{\R}{{\mathbb R}}

\newcommand{\N}{{\mathbb N}}

\newcommand{\T}{T}

\newcommand{\eqE}{{\stackrel{E}{=}}}
\newcommand{\eqD}{{\stackrel{D}{=}}}

\newcommand{\eqDE}{{\stackrel{D+E}{=}}}

\newcommand{\dd}{{\rm d}}

\newcommand{\Id}{  1 \!\!\! \mathrm I}

\newcommand{\qr}{{\rm qr \,}}

\newcommand{\eigh}{{\rm eigh \,}}
\newcommand{\tr}{\mathop{\rm tr \,}}
\newcommand{\diag}{\mathop{\rm diag}}

\newcommand{\rank}{{\mathrm{rank}\,}}

\newcommand{\Nb}{N_{\mathrm{b}}}
\newcommand{\nb}{{n_{\mathrm{b}}}}
\newcommand{\Nr}{{N_{\mathrm{r}}}}

\renewcommand{\ker}{{\rm ker}\,}

\newcommand{\pPi}{\mathnormal{\Pi}}

\newcommand{\Eqn}[1]{Eqn. (\ref{#1})}
\newcommand{\Fig}[1]{Fig. (\ref{#1})}

\newcommand{\textred}[1]{#1}

\newcommand{\textgreen}[1]{#1}

\renewcommand{\nu}{ {n_u}}
\newcommand{\Np}{ {N_p}}

\newcommand{\Nq}{ {N_q}}

\newcommand{\Nm}{ {N_m}}

\begin{document}

\title{Higher-order derivatives of the QR and of the real symmetric eigenvalue decomposition in forward and reverse mode algorithmic differentiation}
\author[1]{Sebastian F. Walter\thanks{sebastian.walter@gmail.com}}
\author[1]{Lutz Lehmann}
\author[1]{ Ren\'e Lamour}
\affil[1]{Department of Mathematics, Faculty of Mathematics and Natural Sciences II,\\
Humboldt-Universit{\"a}t zu Berlin, Rudower Chaussee 25, Adlershof, 12489 Berlin,\\
mail address: Unter den Linden 6, 10099 Berlin, Germany
\\\vspace{6pt} }

\maketitle

\begin{abstract}
We address the task of higher-order derivative evaluation of computer programs that contain QR decompositions and real symmetric eigenvalue decompositions. The approach is a combination of univariate Taylor polynomial arithmetic and matrix calculus in the (combined) forward/reverse mode of Algorithmic Differentiation (AD). Explicit algorithms are derived and presented in an accessible form. The approach is illustrated via examples.
\end{abstract}

{\scriptsize

\textbf{Classification:} 65D25; 12E05; 58C15; 65F15; 65F25
}

\section{Introduction and related work}
This paper is concerned with the efficient evaluation of higher-order derivatives of functions $F: \R^N \rightarrow \R^M$ which are implemented as computer programs that contain numerical linear algebra functions like the QR or the real symmetric eigenvalue decomposition. 

Traditionally, Algorithmic Differentiation (AD) tools like ADOL-C \cite{Griewank1999ACA} or CppAD \cite{Bell:2010} regard the functions defined in the C header file math.h as elementary functions. In the forward mode of AD, their approach to compute higher-order derivatives is to generalize from real arithmetic to univariate Taylor polynomial (UTP) arithmetic  \cite{Griewank1999ACA, Griewank2000EHD, Griewank2008EDP}. For the reverse mode of AD, the program evaluation is traced and stored in a computational graph or on a sequential tape. During the so-called reverse sweep the stored intermediate values are retrieved and used to compute derivatives (c.f. Section \ref{sec:review_ad}).

As explained in Section \ref{sec:computational_model}, the functions in math.h suffice since all computable functions are a concatenation of these functions. However, working only at the expression level has also disadvantages since no global knowledge of the function's structure can be used. A particularly important class of algorithms in science and engineering are numerical linear algebra (NLA) functions. Though NLA functions are typically locally smooth, their implementations often contain non-differentiable operations and program branches. If no special care is taken, this may result in incorrect computations of derivatives. Also, many NLA functions on $\R^{N \times N}$ matrices require $\mathcal O(N^3)$ arithmetic operations. Since during the reverse mode intermediate results are required, this would yield an $\mathcal O(N^3)$ memory requirement. Though it may be possible to adapt codes to yield reduced memory requirements, as for instance reported for the LU decomposition \cite{Griewank2003AMV}, in practice it can be a cumbersome and error-prone process. Also one would like to reuse existing, high-performance implementations of NLA algorithms. Adding the NLA functions to the list of elementary functions circumvents this problem. This has been realized before \cite{Bucker2002HAf, Bischof1996HAt} and also UTP algorithms for some NLA functions (e.g. the solution of linear equations) have been implemented in software \cite{ETP03}.

The contribution of this paper is to provide explicit algorithms for UTP arithmetic applied to the QR decomposition and the real symmetric eigenvalue decomposition. Note that our approach to the real symmetric eigenvalue decomposition is similar to \cite{andrew:98, vanderaa:06} but our algorithmic result differs. In addition, we also treat the reverse mode of AD.

The document is structured as follows: In Section \ref{sec:motivating_example} we give two application examples for the algorithms presented in this document, followed by a brief review of the underlying computational model in Section \ref{sec:computational_model}. We shortly describe the basics of AD in Section \ref{sec:review_ad} where we make use of the results from \ref{sec:computational_model}. In Section \ref{sec:utpm} we describe the general approach of NLA functions. After that, we apply the results from Section \ref{sec:review_ad} to find extended functions for the QR and eigenvalue decomposition in Section \ref{sec:qr} and \ref{sec:eigh} and also provide pullback algorithms that are necessary in in the reverse mode of AD. Finally, we present some numerical results in Section \ref{sec:numerical_results}.

\section{Application examples for the proposed algorithms} \label{sec:motivating_example}
The purpose of this section is to show two application examples where higher-order derivatives of computer programs that contain the QR and the real symmetric eigenvalue decomposition are necessary. 

\subsection{Optimum experimental design}
The goal in optimum experimental design (OED) is to minimize some cost function representing the size of the confidence region of parameters of interest. We consider here a popular formulation  where the objective function $\Phi(C) \in \R$ depends on the covariance matrix $C \in \R^{\Np \times \Np}$ of a constrained parameter estimation problem, where the covariance matrix is computed by
\begin{eqnarray}
C &=& (\Id, 0)
\begin{pmatrix}
J_1^T J_1 & J_2^T \\
J_2 & 0 \\
\end{pmatrix}^{-1}
\begin{pmatrix}
\Id\\
0
\end{pmatrix} \;, \label{eqn:covariance_matrix}
\end{eqnarray}
and we assume that $J_1 \equiv J_1(q) \in \R^{\Nm \times \Np}$, $J_2 \equiv J_2(q) \in \R^{\Nr \times \Np}$, $\Id \in \R^{N_p \times N_p}$ and $q \in \R^{\Nq}$. The notation is motivated as follows: $p \in \R^\Np$ are model (pseudo-)parameters, $q \in \R^{\Nq}$ are control variables and $J_1$ and $J_2$ are Jacobians  of the residuals resp. of the constraint function with respect to the parameters $p$. Typical choices for cost function $\Phi$ are the trace, the determinant or the maximum eigenvalue of the covariance matrix $C$. Though \Eqn{eqn:covariance_matrix} correctly describes the covariance matrix $C$, the actual algorithmic implementation is often a code like
\begin{eqnarray}
C &=& Q_2^T \left( Q_2 J_1^T J_1 Q_2^T \right)^{-1} Q_2 \label{eqn:cov_qr} \;,
\end{eqnarray}
where $Q_2$ results from a QR-like decomposition of $J_2$, i.e.
$J_2^T = (Q_1^T, Q_2^T)
(L, 0)^T
$. The matrices $J_1$ and $J_2$ are assumed to satisfy the constraint qualification $\rank(J_2)) = \Nr$ and the condition $\rank (J) = \Np$, where $J = (J_1^T, J_2^T)^T$. The matrix $Q_2$ spans the nullspace of $J_2$. For a detailed discussion we refer to K\"orkel \cite{Koerkel:02} and to Bock and Kostina \cite{BKKS07}.

Newton-type optimization algorithms require at least the gradient $\nabla_q \Phi(q)$ of the objective function $\Phi$. To obtain good convergence near the local minimizer, it is often advantageous if exact second-order derivatives are available.
Since the number of controls $\Nq$ can be large, one would like to have the possibility to compute these derivatives in the reverse mode of AD. Robust objective functions are often formulated in a way that require third and even higher-order derivatives, so it is necessary to have algorithms that scale easily to arbitrary order. 

Thus, this example requires the differentiation of the nullspace of a matrix, the matrix product, matrix inversion, the QR decomposition and the objective function evaluation, e.g. the eigenvalue decomposition.

\subsection{Index determination of differential algebraic equations}
Many dynamical problems in chemical engineering, rigid body mechanics, circuit simulation and control theory are described by Differential Algebraic Equations (DAEs) of the form
\begin{equation}
0 = f\left( \tfrac{\dd }{\dd x}d(y,x), y, x\right),
\;x\in I=[a,b]\subset\R,
\end{equation}
where $y:\R\to\R^m$ lives in suitable function space, 
$f:\R^n\times \R^m\times I\to\R^m$, $d:\R^m\times I\to\R^n$ 
are sufficiently smooth and typically $n$ is smaller than $m$.

Using higher--order derivatives of the functions in the DAE one can, in general, transform the DAE system into an ODE system of order one. The \emph{differentiation index} is the highest derivative order required in this process, that is, derivatives of up to this order of the original equations are part of any solution of the DAE. The knowledge of the index allows to estimate the difficulty to solve the DAE.

There are many different index definitions. Here we consider the tractability index. To compute it, the DAE is linearized along a given function $\bar y(x)$ as
\[
 \underbrace{
 	\tfrac{\partial}{\partial z}f(\bar w(x), \bar y(x), x)
 }_{
	=A(x)\in \R^{m \times n}
 }
 \tfrac{\dd }{\dd x} 
 \Bigl( 
	\underbrace{
		\tfrac{\partial}{\partial y}d(\bar y(x),x)
	}_{
		=D(x)\in \R^{n \times m}
	}
	z(x)
 \Bigr) +
 \underbrace{
 	\tfrac{\partial}{\partial y}f(\bar w(x), \bar y(x),x)
 }_{
 	=B(x)\in \R^{m \times m}
 }
 z(x) =
 \underbrace{
 	-f\left( \bar w(x),\bar y(x),x\right)
 }_{
 	q(x)
 }
\]
with $\bar w(x)=\frac{\dd}{\dd x}d(\bar y(x), x)$. The coefficient functions $A = A(x),\;D = D(x)$ and $B = B(x)$ give rise to a matrix sequence
\begin{align}
 G_0&=AD,&B_0&=B,\nonumber\\
G_{i+1}&=G_i+B_iQ_i,&
B_{i+1}&=B_iP_i-G_{i+1}D^-\tfrac{\dd}{\dd x}(DP_0\cdots P_{i+1}D^-)DP_0\cdots P_{i},\label{Bip1}
\end{align}
where $Q_i$ describes a projector onto $\ker G_i$, $P_i=I-Q_i$ and $D^-$ is a generalized reflexive inverse of $D$. Now, the tractability index is the smallest number $\mu \in \N $ where $G_{\mu}$ is nonsingular.
The projectors $Q_i$ can be determined mainly by use of a QR decomposition.

A QR decomposition of the potentially singular matrix $G \in  \R^{M \times M}$ with $\rank G = r$ results in 
\[
	G \pPi = Q \begin{pmatrix}
	            R_1 & R_2\\
		    0 & 0
	           \end{pmatrix},
\]
where $\pPi$ describes a column permutation matrix, $Q$ an orthogonal matrix and $R_1 \in \R^{r \times r}$ an 
upper triangular matrix. Then a nullspace projector $Q_G$ onto $\ker G$ is given by
\[
 Q_G=\pPi \begin{pmatrix}
           0 & -R^{-1}_1R_2\\
	    0& I
          \end{pmatrix}\pPi^T.
\]
The computation of $B_{i+1}$ via \eqref{Bip1} needs the differentiation of $DP_0\cdots P_{i+1}D^-$ with respect to $x$.
Thus, higher--order derivatives of a function that contains the QR decomposition are necessary.
For a in-depth discussion of index definition of DAEs see M{\"a}rz \cite{Mae02,Mae03}.

\section{Computational model}\label{sec:computational_model}
We consider functions 
\begin{eqnarray*}
F: \R^N &\rightarrow& \R^M \\
 x & \mapsto&  y = F(x) \;,
\end{eqnarray*} that can be  described by the \emph{three-part form}
\begin{eqnarray*} \label{eqn:seqelem}
v_{n-N} & = & x_n \quad\quad\quad n=1,\dots,N\\
v_{l} & = & \phi_l(v_{j \prec l}) \quad l =1,\dots, L\\
y_{M-m} & = & v_{L-m} \quad\quad m = M-1, \dots, 0 \;,
\end{eqnarray*}
where  $\phi_l \in \{+,-,\cdot, /,\sin, \exp, \dots \}$ are called \emph{elementary functions}, $v_l$ are intermediate values and $v_{i \prec l}$ denote the tuples of input arguments of $\phi_l$.  For instance the function $F:\R^2 \rightarrow \R$, $\textgreen{x} \mapsto \textgreen{y} = F(\textgreen{x}) = \textred{\sin}( \textgreen{x_1} \textred{+} \textred{\cos}(\textgreen{x_2}) \textred{*} \textgreen{x_1})  $ is described by
\begin{center}
\begin{tabular}{|c|l c l c l|}
\hline
independent & \textgreen{$v_{-1}$} &=& \textgreen{$ x_1 $} &=& $  3 $\\
independent & \textgreen{$v_{0}$}  &=& \textgreen{$ x_2 $} &=& $ 7  $\\
\hline
&\textgreen{$v_{1} $} &=& $ \textred{\phi_1}( \textgreen{v_{0}} )$ &=& $\cos(v_{0})$ \\
&\textgreen{$v_{2} $} &=& $ \textred{\phi_2}( \textgreen{v_{1}}, \textgreen{v_{-1}} )$ &=& $ v_1 v_{-1}$ \\
&\textgreen{$v_{3} $} &=& $ \textred{\phi_3}( \textgreen{v_{-1}}, \textgreen{v_{2}} )$ &=& $  v_{-1} + v_{2}$ \\
&\textgreen{$v_{4} $} &=& $ \textred{\phi_4}( \textgreen{v_{3}} )$ &=& $  \sin(v_{3})$ \\
\hline
dependent & \textgreen{$y$} &=& \textgreen{$v_4$} && \\
\hline
\end{tabular}
\end{center}
It shows a sequential representation of the computation. Alternatively, one can describe the function evaluation as composite function
\begin{eqnarray}
F(x) &=& P_y \circ \Phi_L \circ \Phi_{L-1} \circ \dots \circ \Phi_1 \circ P_x^T (x) \;,
\end{eqnarray}
where $\Phi_l: \mathcal H \rightarrow \mathcal H$, $s^{(l-1)} \mapsto s^{(l)} = \Phi_l(s^{(l-1)} )$ are called \emph{elementary transitions} that operate the \emph{state space} $\mathcal H$. 
Each elementary transition can be written as
\begin{equation}
\Phi_l = P_l \circ  \phi_l \circ Q_l + ( \Id - (1 - \sigma_l) P_l \circ P_l^T) \;.
\end{equation}
where the functions $\phi_l: \mathcal D_l \subseteq \mathcal H_l \rightarrow \mathcal H_l \in \{+,-,*,/,\sin,\exp,\dots \}$ are the elementary functions. The $Q_l: \mathcal H \rightarrow \mathcal H_l$ map to the domains of the elementary functions and the $P_l: \mathcal H_l \rightarrow \mathcal H$ map back to the overall state space. The functions $P_x^T$ and $P_y$ are used to map the independent variables $x$ to the state $s^{(0)}$ and $s^{(L)}$ to $y$. The case $\sigma_l = 1$ corresponds to an augmented assignment $s_l = s_l + \phi_l(s_l)$ and $\sigma_l = 0$ to the usual assignment $s_l =  \phi_l(s_l)$.  For our purposes it suffices to consider a real vector space as state space, i.e., the mappings $P_l$ and $Q_l$ can be written as matrices. For a more detailed discussion see Griewank \cite{Griewank2003AMV}.

\section{Algorithmic differentiation}\label{sec:review_ad}
In this section we briefly review some key results from the theory of AD that will be necessary in Section \ref{sec:qr} and \ref{sec:eigh}. For a detailed discussion we refer to the standard reference ``Evaluating Derivatives'' by Griewank and Walther \cite{Griewank2008EDP}. 

\subsection{The forward mode}
One can use univariate Taylor series expansions to compute higher-order (directional) derivatives. The basic observation is that given a smooth curve $x(t) = x_0 + x_1 t$ with $t \in (- \epsilon, \epsilon)$, $\epsilon > 0$, and a smooth function $F$ one obtains a smooth curve $y(t) = F(x(t))$ with the Taylor series expansion
\begin{equation}
y(t) =  \sum_{d=0}^{D-1} y_d t^d + \mathcal O(t^D) = \sum_{d=0}^{D-1} \left. \frac{1}{d!} \frac{\dd^d}{\dd t^d} F(x(t)) \right|_{t=0} t^d + \mathcal O(t^D) \;.  
\end{equation}
By application of the chain rule one can interpet the terms of the expansion. The zeroth derivative is the normal function evaluation $y_0 = F(x_0)$ and the first coefficient $y_1 = \left. \frac{\dd }{\dd t} F(x(t)) \right|_{t=0} = \frac{\dd}{\dd x}F(x_0) \cdot x_1$ is a directional derivative.

In the context of AD it is advantageous to generalize the notion of Taylor series expansions to a purely algebraic task. In other words, for arithmetic with univariate Taylor polynomials (UTP) one extends functions $F:\R^N \rightarrow \R^N$ to functions $E_D(F): \R^N[T]/(T^D) \rightarrow \R^M[T]/(T^D)$. We denote representing elements of the polynomial factor ring $\R^N[T]/(T^D)$ as 
\begin{equation}
[x]_D := [x_1,\dots, x_{D-1}] := \sum_{d=0}^{D-1} x_d T^d \;,
\end{equation}
where $x_d \in \R^N$ is called \emph{Taylor coefficient}. The quantity $T$ is an indeterminate, i.e., a formal variable. The \emph{extended function} $E_D(F)$ is defined by its action
\begin{eqnarray}
[y]_D &=& E_D(F)([x]_D ) = \sum_{d=0}^{D-1} y_d \T^d = \left. \sum_{d=0}^{D-1} \frac{1}{d!} \frac{\dd^d}{\dd t^d} F(\sum_{d=0}^{D-1} x_d t^d) \right|_{t=0} \T^d \;.
\end{eqnarray}
The notation $[x]_D \equiv [x]_{:D-1}$ and $[x]_{d+1:D-1} \equiv [x]_{d+1:} \equiv [x_{d+1}, \dots, x_{D-1}]$ will be useful later on.
One can show that this definition is compatible with the usual polynomial addition and multiplication. Furthermore, any composite function $F(x) = (H \circ G)(x) = H(G(x))$ satisfies 
\begin{eqnarray}
E_D(F) &=& E_D(H) \circ E_D(G) \;.
\end{eqnarray}
I.e., the extension function $E_D$ is a homomorphism which preserves function composition. An immediate consequence is that it is necessary to find algorithms only for the very limited set of elementary functions $\phi \in \{ +, -, *, /, \sin, \cos, \exp, \dots \}$. Explicitly, one performs the program transformation $E_D(F) = E_D(\Phi_L) \circ \dots \circ E_D(\Phi_1) ([x]_D)$. We call the action of computing $[y]_D = E_D(F)([x]_D)$, i.e., the resolution of the symbolic dependence to obtain the numerical value $[y]_D$, the \emph{pushforward} of the function $E_D(F)$.

Many functions are implicitly defined by equations of the type $0 = F(x,y) \in \R^M \;$,
where $x \in \R^N$ are the inputs and $y \in \R^M$ the outputs. 
The idea is to demand that the \emph{defining equations of order $D$}
\begin{equation}
0 \eqD E_D(F)([x]_D, [y]_D)
\end{equation}
 should be satisfied. By $\eqD$ it is meant that $[x] \eqD [y]$ if $x_d = y_d$ for $d = 0,\dots, D-1$. This is also often written either as $[x] = [y] + \mathcal O(T^D)$ or  $[x] = [y] \mod T^D$.
The defining equations lead directly to an algorithmic approach to compute $[y]_D$, the so-called \emph{Newton-Hensel lifting}. In the literature it is often also just called Hensel-lifting or Newton's method \cite{Griewank2008EDP}. Assuming  $[y]_D$ is already known and satisfies $0  \eqD  E_D(F)([x], [y]_D)$, one can lift the computation to a higher degree. Explicitly, one tries to solve $ 0  \eqDE  E_{D+E}(F)([x], [y]_{D+E})$. Splitting $[y]_{D+E} = [y]_D + [\Delta y]_E T^D$ and performing a first order Taylor expansion of $F$ about $[y]_D$ yields after a short calculation
\begin{eqnarray}
[\Delta y]_E &\eqE & - [F_y]_E^{-1} [\Delta F]_E \;, \label{eqn:newton_hensel}
\end{eqnarray} 
where $E_{D+E}(F)([x], [y]_{D}) \eqDE [\Delta F]_E T^D$ and $[F_y]_E := E_E(\frac{\dd F}{\dd y})([x], [y]_{E})$.
Setting $E = D$ means that at each step the number of correct coefficients is doubled. In this case we call it \emph{Newton's method}. In the case $E=1$ only the next coefficient is computed. We call the special case $E=1$ \emph{sequential Hensel lifting} which is also the formula that is often given as part of the implicit function theorem. The difference is that Newton-Hensel lifting is a purely algebraic task. For a discussion on how to obtain asymptotically fast algorithms and for the nomenclature see e.g. Bernstein \cite{bernstein:2001,Bernstein:08}.

\subsection{The reverse mode}
The basic idea of the reverse mode of AD is to pullback linear forms $\alpha$ to obtain an explict mapping $\bar y \mapsto \bar x$. I.e., given $F: \R^N \rightarrow \R^M, \;y = F(x)$ one has
\begin{eqnarray}
\alpha(\bar y, y) &=& \sum_{m=1}^M \bar y_m \dd y_m  = \sum_{m=1}^M \bar y_m \sum_{n=1}^N \frac{\partial F_m}{\partial x_n} \dd x_n = \sum_{n=1}^N \bar x_n \dd x_n = \alpha( \bar x, x) \;,
\end{eqnarray}
where $\bar x_n = \sum_{m=1}^M \bar y_m \frac{\partial F_m}{\partial x_n}$.
For notational reasons one uses $\sum_{n=1}^N \bar x_n \dd x_n \equiv \bar x^T \dd x$. We call the action of going back one level of the symbolic dependence the \emph{pullback} of the linear form $\alpha(\bar y, y)$. For a more detailed discussion on calculations with differentials see Magnus and Neudecker \cite{Magn:Neud:1999}.

It is also possible to compute higher-order derivatives by combining UTP arithmetic and the reverse mode of AD. For that, the UTP algorithms are interpreted as functions mapping $D$ coefficients $x_d$,  $0 \leq d < D$ to $D$ coefficients $y_d$, $0 \leq d < D$, i.e., a mapping from $\R^{N \times D} 
\rightarrow \R^{M \times D}$ with a special structure. One can formally define a linear form by 
\begin{equation}
E_D(\alpha)( [\bar y]_D , [y]_D) := [y]_D^T \dd [y]_D \;. 
\end{equation}
 Here, $\dd [y]_D = \sum_{d=0}^{D-1} \dd y_d T^d$ is a formal polynomial where each coefficient is a differential and $[\bar y]_D^T \dd [y]_D = \sum_{m=1}^M [\bar y_m]_D \dd [y_m]_D$ computes the polynomial multiplication of formal polynomials. One can show that 
\begin{eqnarray}
E_D(\alpha)([\bar y]_D, [y]_D) &\eqD& [\bar y]_D^T E_D( \frac{\partial F}{\partial x})([x]_D) \dd [x]_D \eqD  [\bar x_n]_D^T \dd [x_n]_D  = E_D(\alpha)([\bar x]_D, [x]_D)
\end{eqnarray}
holds \cite{Christianson1991RAa}. One can interpret this result as follows: If $[\bar y]_D = w \in \R^{M}$ then $[\bar x]_D = E_D( w^T \frac{\partial F}{\partial x})([x]_D)$. Setting $w = e_i$ a Cartesian basis vector would yield the Taylor expansion of the $i$'th row of the Jacobian. The interpretation of the Taylor coefficients as derivatives yields higher-order derivatives. If $M=1$ and $w = 1$ one obtains the Taylor expansion of the gradient $ [\bar x]_D = E_D(\nabla F)([x]_D)$. E.g., propagating the UTP $[x]_2 = x_0 + x_1 T$ would yield $[\bar x]_2 = \bar x_0 + \bar x_1 T$ where $\bar x_0 = \nabla_x F(x)$ and $\bar x_1 = \nabla_x^2 F(x) \cdot x_1$, i.e., a Hessian-vector product.

\section{Defining equations of numerical linear algebra functions} \label{sec:utpm}
As briefly mentioned in the introduction, Numerical Linear Algebra (NLA) functions can be viewed as algorithms representing a concatenation of functions like $+,-,*,/,\sin,\cos, \dots$ and thus it is possible to apply the AD techniques described in the previous section directly to the algorithm. However, there is also the possibility to regard the problem from a more abstract point of view. Many NLA functions are implicitly defined by a system of equations. 

For instance the QR decomposition is defined by the defining equations
\begin{eqnarray}
0 &=& Q R - A \\
0 &=& Q^T Q - \Id \\
0 &=& P_L \circ  R \;,
\end{eqnarray}
where $A,R \in \R^{M \times N}$ with $M \geq N$  and $Q \in \R^{M \times M}$.
The functional dependence of the defining equations is denoted 
\begin{eqnarray}
Q, R = \qr(A) \;.
\end{eqnarray}
 Only the first $N$ rows $R_{:N,:} \in \R^{N \times N}$ of $R$ are nonzero. For convenience reasons we use the slicing notation $i:j = (i,i+1,\dots, j)$.

The defining equations of the symmetric eigenvalue decomposition are given by
\begin{eqnarray}
0 &=& Q^T A Q- \Lambda \\
0 &=& Q^T Q - \Id \\
0 &=& (P_L + P_R) \circ  \Lambda \;,
\end{eqnarray}
where $A \in \R^{M \times M}$ is symmetric.
The functional dependence is denoted $ \Lambda, Q = \eigh(A)$.
We call the matrices $(P_L)_{ij} = \delta_{j<i}$ and $(P_R)_{ij} = \delta_{i<j}$ skeletal projectors since their elementwise product with a matrix returns strictly lower resp. strictly upper triangular matrices.

\section{The QR decomposition}\label{sec:qr}
Before we derive algorithms based on the defining equations, we briefly investigate what can go wrong if a typical implementation of the QR decomposition using Householder reflections is evaluated in UTP arithmetic. Consider Algorithm 5.1.1 from the book ``Matrix Computations'' by Golub and Van Loan \cite{golub:1996} which we adapted to our notation  in Algorithm \ref{algo:golub_house}. From the AD point of view, the problematic part in the code is the check $\sigma = 0$. Since a paradigm of AD tools is that the control flow must remain unchanged, the check $\sigma =0 $ only considers the zeroth coefficient $x_0$ of a UTP. Hence, if $[x]_2 = e_1 + x_1 T$ is given as input and $x_1 \neq 0$, the algorithm will simply evaluate $\beta = 0$ and return. As final result, one obtains a matrix $[R]_2$ where $R_1$ is not upper triangular. The LAPACK implementation  (LAPACK-3.2.2) of DGEQRFP.f calls the subroutine DLARFGP.f which contains a similar check. Hence, automatic augmentation based on AD principles can go wrong in such cases. 

As a side remark, note that additionally the function realized by this algorithm has a pole at $\sigma=0$, producing numerical overflow for $\sigma\approx0$.
\begin{algorithm2e}[!h]
\SetKwInOut{Input}{input}\SetKwInOut{Output}{output}
\SetLine 
\Input{$x \in \R^N$ }
\Output{$v \in \R^N$ with $v_1 = 1$}
\Output{$\beta \in \R$ }
\BlankLine
$\sigma = x_{2:}^T x_{2:}$\\
$v = \begin{pmatrix}
      1\\ x_{2:}
     \end{pmatrix}
$\\
\eIf{$\sigma = 0$}{
$\beta = 0$
}
{
$\mu = \sqrt{ x_1^2 + \sigma} $\\
\eIf{$x_1 \leq 0$}{
$v_1 = x_1 - \mu$
}
{
$v_1 = - \sigma/(x_1 + \mu)$
}
$\beta = 2 v_1^2/(\sigma + v_1^2)$ \\
$v = v/v_1$
}
\caption{Householder Vector. The reflector is $P_v=I-\beta vv^\top$, with $v_1=1$.}\label{algo:golub_house}
\end{algorithm2e}

\subsection{Pushforward in Taylor arithmetic}
We now come to the higher--level approach that is based on the defining equations given in Section \ref{sec:utpm}. To compute $[Q]_D, [R]_D = E_D(\qr)([A]_D)$ one can apply Newton-Hensel lifting to solve
\begin{eqnarray}
0 &\eqD& [Q]_D [R]_D - [A]_D \\
0 &\eqD& [Q]_D^T [Q]_D - \Id \\
0 &\eqD& P_L \circ  [R]_D \;.
\end{eqnarray}
The direct application of \Eqn{eqn:newton_hensel} should be avoided since $F_y$ is sparse and has a lot of structure. Rather, one assumes that one has already computed $[Q]_D$ and $[R]_D$ and computes the next $1 \leq E \leq D$ coefficients by performing a first order Taylor expansion $[Q]_{D+E} = [Q]_D + [\Delta Q]_E T^D$ and $[R]_{D+E} = [R]_D + [\Delta R]_E T^D$ and tries to solve for the yet unknown $[\Delta R]_E$ and $[\Delta Q]_E$. As result one obtains Proposition \ref{prop:qr}. For convenience, we use the convention that $R_{d; i,j}$ is the $i$'th row and $j$'th column of the $d$'th coefficient of $[R]_D$. 
\begin{proposition} \label{prop:qr}
Let $[A]_{D+E} \in \R^{M \times N}[T]/(T^{D+E})$ with $M \geq N$ and $1 \leq E \leq D$, $[R]_D \in   \R^{M \times N}[T]/(T^D)$ where $[R_{:N,:}]_D$ is upper triangular with nonsingular $R_{0;:N,:}$ and $[Q]_D \in  \R^{M \times M}[T]/(T^D)$ orthogonal be given and satisfy the defining equations of order $D$. Then $[\Delta R_{:N,:}]_E \equiv [R_{:N,:}]_{D:D+E-1}$ and $[\Delta Q]_E \equiv [Q]_{D:D+E-1}$ are given by
\begin{eqnarray}
\; [\Delta F]_E T^D &\eqDE& - [Q]_D [R]_D + [A]_{D+E} \\
\; [\Delta G]_E T^D &\eqDE& - [Q]^T_D [Q]_D + \Id \\
\; [S]_E &\eqE&  \frac{1}{2} [\Delta G]_E \\
\; P_L \circ ([X_{:,:N}]_E) & \eqE & P_L \circ \left( [Q]^T_E [\Delta F]_E [R_{:N,:}]_E^{-1} \right) - P_L \circ [S_{:,:N}]_E \\
\; [\Delta R]_E & \eqE & [Q]_E^T [\Delta F]_E - ([S]_E + [X]_E) [R]_E \\
\; [\Delta Q]_E & \eqE & [Q]_E \left( [S]_E + [X]_E \right)\;,
\end{eqnarray}
where $P_L \in \R^{M \times N}$ with $(P_L)_{ij} = \delta_{j < i}$.
\end{proposition}
\begin{proof}
In the Appendix \ref{appendix:utpm_qr}.
\end{proof}

\subsection{Pullback}
\begin{proposition} \label{prop:pullback_qr}
Let $A, R, \bar R \in \R^{M \times N}$ resp. $Q,\bar Q \in \R^{M \times M}$ be given and it holds $M \geq N$, $\rank(A) = N$, $Q,R = \qr(A)$. Then $\bar A \in \R^{M \times N}$ can be computed by
\begin{eqnarray}
\bar A &=&  Q \left( \bar R + \left( P_L \circ \left(  R \bar R^T - \bar R R^T  + Q^T \bar Q - \bar Q^T Q \right) \right)R^{+T} \right)\;.
\end{eqnarray}
Here, $R^+$ denotes the Moore-Penrose pseudoinverse of $R$. That means it satisfies $R R^+ R = R$ and since $R$ has full column rank also $R^+ R = \Id$.
\end{proposition}
\begin{proof}
In Appendix \ref{appendix:reverse_qr}.
\end{proof}

\subsection{Explicit algorithms}
One can use Proposition \ref{prop:qr} to derive an explicit algorithm as shown in Algorithm \ref{algo:forward_qr}, where at each step $E=1$ is used.

\begin{algorithm2e}[!h]
\SetKwInOut{Input}{input}\SetKwInOut{Output}{output}

\SetLine 
\Input{$[A]_{D} = [A_0, \dots, A_{D-1}]$, where $A_d \in \R^{M \times N}$, $d=0,\dots,D-1$ and $\rank(A_0) = N$, $M \geq N$. }
\Output{$[Q]_{D} = [Q_0, \dots, Q_{D-1}]$ matrix with orthonormal column vectors, where $Q_d \in \R^{M \times N}$ , $d=0,\dots,D-1$ }
\Output{$[R]_{D} = [R_0, \dots, R_{D-1}]$ upper triangular, where $R_d \in \R^{N \times N}$ , $d=0,\dots,D-1$ }
\BlankLine
$Q_0, R_0 = \qr(A_0)$\\
\BlankLine
\For{$d = 1$ \KwTo $D-1$}{
$\Delta F =  A_{d} - \sum_{k=1}^{d-1} Q_{d-k} R_k $\\
$S =  - \frac{1}{2} \sum_{k=1}^{d-1} Q^T_{d-k} Q_k $\\
$X_{:,:N} = P_L \circ ( Q_0^T \Delta F R_{0;:N,:N}^{-1} - S_{:,:N})$ \\
$X_{:,N+1:} = 0$ \\
$X = X - X^T $\\
$R_d = Q_0^T \Delta F - (S + X) R_0$ \\
$Q_d = Q_0 ( S + X )$
}
\caption{ Sequential Hensel lifting for the QR decomposition.}\label{algo:forward_qr}
\end{algorithm2e}

The pullback can be computed in Taylor arithmetic. In the global derivative accumulation it is necessary to update the value of $[\bar A]_D$. This happens if $[A]_D$ is input of more than one function. The algorithm for the pullback takes this into consideration.

\begin{algorithm2e}[!h]
\SetKwInOut{Input}{input}\SetKwInOut{Output}{output}\SetKwInOut{Inout}{input/output}

\SetLine 
\Input{$[A]_{D} = [A_0, \dots, A_{D-1}]$, where $A_d \in \R^{M \times N}$, $d=0,\dots,D-1$, $M \geq N$. }
\Input{$[Q]_{D} = [Q_0, \dots, Q_{D-1}]$, where $Q_d \in \R^{M \times M}$ , $d=0,\dots,D-1$ }
\Input{$[R]_{D} = [R_0, \dots, R_{D-1}]$, where $R_d \in \R^{M \times N}$ , $d=0,\dots,D-1$ }
\Inout{$[\bar A]_{D} = [\bar A_0, \dots, \bar A_{D-1}]$, where $\bar A_d \in \R^{M \times N}$, $d=0,\dots,D-1$, $M \geq N$. }
\Input{$[\bar Q]_{D} = [\bar Q_0, \dots, \bar Q_{D-1}]$, where $\bar Q_d \in \R^{M \times M}$ , $d=0,\dots,D-1$ }
\Input{$[\bar R]_{D} = [\bar R_0, \dots, \bar R_{D-1}]$, where $\bar R_d \in \R^{M \times N}$ , $d=0,\dots,D-1$ }
\BlankLine
{
\begin{eqnarray*}
 [\bar A]_D &=& [\bar A]_D +  [Q]_D \cdot \\
&& \cdot \left( [\bar R]_D + \left( P_L \circ \left(  [R]_D [\bar R]_D^T - [\bar R]_D [R]_D^T  + [Q]_D^T [\bar Q]_D - [\bar Q]_D^T [Q]_D \right) \right) [R]_D^{+T} \right) 
\end{eqnarray*}}
\caption{Pullback of the QR decomposition in Taylor arithmetic. The inputs $[A]_D, [Q]_D, [R]_D$ must satisfy the defining equations.}
\end{algorithm2e}

\pagebreak[6]
\section{The real symmetric eigenvalue decomposition}\label{sec:eigh}
The problem of finding eigenvalues and eigenvectors arises in a wide variety of practical applications. As for the QR decomposition, we want to have algorithms that compute the real symmetric eigenvalue decomposition in UTP arithmetic as well as pullback algorithms. The symmetric eigenvalue decomposition is also important since the Singular Value Decomposition (SVD) of real matrices is closely related to it. More explicitly, one can compute the SVD of a matrix $A \in \R^{M \times N}$ of rank $r$., i.e., $A = U \Sigma V^T$, where $\Sigma = \diag( \Sigma_1, 0)$, $U = (U_1, U_2)$, $U_1 \in \R^{M \times r}$, $V = (V_1, V_2)$, $V_1 \in \R^{N \times r}$ as
\begin{eqnarray*}
C = \begin{pmatrix}
     0 & A \\ A^T & 0
    \end{pmatrix}
= P^T 
\begin{pmatrix}
\Sigma_1 & 0 & 0 \\
0 & - \Sigma_1 & 0 \\
0 & 0 & 0 \\
\end{pmatrix}
P \;,
\end{eqnarray*}
where
\begin{eqnarray*}
P = \frac{1}{\sqrt{2}} 
\begin{pmatrix}
U_1 & U_1 & \sqrt{2} U_2 & 0 \\
V_1 & - V_1 & 0 & \sqrt{2} V_2
\end{pmatrix}^T
\end{eqnarray*}
is orthogonal \cite{bjor:96}.

\subsection{Pushforward in Taylor arithmetic}
Given the symmetric polynomial matrix $[A]_D \in \R^{N\times N}[T]/(T^D)$. The eigenvalue decomposition is the solution $[\Lambda]_D, [Q]_D  \in \R^{N\times N}[T]/(T^D)$ of the implicit system
\begin{eqnarray}
0 &\eqD& [Q]_D^T [A]_D [Q]_D - [\Lambda]_D \\
0 &\eqD& [Q]_D^T [Q]_D - \Id \\
0 &\eqD& (P_L + P_R) \circ  [\Lambda]_D \;,
\end{eqnarray}
which is called the \emph{defining equations of order $D$}. We also assume that the eigenvalues are sorted as $[\Lambda_{11}]_D \leq [\Lambda_{22}]_D \leq \dots \leq [\Lambda_{NN}]_D$.
The functional dependence is denoted
\begin{eqnarray}
[\Lambda]_D, [Q]_D & = \eigh([A]_D) \;.
\end{eqnarray}
Let  $\Lambda, Q = \eigh(A)$ be the usual symmetric eigenvalue decomposition. We denote the diagonal of $[\Lambda]_D$ as $[\lambda]_D = \diag([\Lambda]_D)$. If eigenvalues are repeated, i.e., multiple, the eigenvectors generalize to eigenspaces and the columns of $Q$, that are associated to such a multiple eigenvalue, are not unique. Rather, any orthonormal basis could be the result. This has consequences for the Hensel-Newton lifting approach, because given $[Q]_D$ and $[R]_D$ that satisfy the defining equations of order $D$ it is generally not possible to find a $[\Delta Q]_E$ and $[\Delta R]_E$ such that $[Q]_{D+E} = [Q]_D + [\Delta Q]_E T^D$ and $[R]_{D+E} = [R]_D + [\Delta R]_E T^D$ satisfy the defining equations of order $D+E$. The higher-order coefficients $[\Delta A]_E$ enforce additional conditions on the chosen basis of the eigenspaces.
A wrong choice of $[Q]_D$ means that $0 \eqDE (P_L + P_R) \circ  [\Lambda]_{D+E}$ cannot be satisfied. However, $0 \eqD P_b^{D} \circ  [\Lambda]_{D+E}$ can be satisfied.
The matrix $P_b^D$ is a skeletal projector with zero blocks on the main diagonal whose size corresponds to the multiplicity of an eigenvalue $[\lambda]_D$ and all other entries are ones. 
The \emph{multiplicity} $m^d([\lambda_{j}]_D)$ of an eigenvalue $[\lambda_{j}]_D$ of \emph{level} $d$ is defined to be the number of $i \in \N$ s.t. $[\lambda_{j}]_D \stackrel{d}{=} [\lambda_{i}]_D$. I.e., 
\begin{eqnarray*}
\diag([\Lambda]_d) &=& ( \underbrace{[\lambda_{1}]_d, \dots, [\lambda_{1}]_d}_{m^d([\lambda_{1}]_D) \mbox{ times}}, \dots, \underbrace{[\lambda_{\Nb^d }]_d, \dots, [\lambda_{\Nb^d }]_d}_{m^d([\lambda_{\Nb^d }]_D) \mbox{ times}} ),
\end{eqnarray*}
 where $\Nb^d$ is the number of different eigenvalues at level $d$.
We define $b^d \in \N^{\Nb^d + 1}$ to be a vector satisfying $ m^d([\lambda_{\nb }]_D) = b_{\nb+1}^d - b_{\nb}^d$. The symbol $b$ is used because it relates to blocks in the matrix. The elements of $P_b^d$ satisfy  $(P^d_b)_{ij} = 1 - \sum_{\nb =1}^{\Nb^d +1} \delta_{ b_\nb^d \leq i < b_{\nb+1}^d}  \delta_{ b_\nb^d \leq j < b_{\nb+1}^d}$. This notation is a little cumbersome but turns out to be helpful. One defines $b^0 = [0,N+1]$. The vector $b^1$ represents the multiplicities of the usual symmetric eigenvalue decomposition. E.g., for $N = 3$ and $b^d=[1,3,4]$ one has
\begin{eqnarray*}
P_b^d =
\begin{pmatrix}
0 & 0 & 1 \\
0 & 0 & 1 \\
1 & 1& 0
\end{pmatrix}
 \;.
\end{eqnarray*}
We reformulate the overall problem as a sequence of subproblems. We call the implicit system
 \begin{eqnarray}
0 &\eqD& [Q^d]_D^T [A]_D [Q^d]_D - [\Lambda^{d}]_D \\
0 &\eqD& [Q^d]_D^T [Q^d]_D - \Id \\
0 &\stackrel{d}{=}& (P_L + P_R) \circ  [\Lambda]_{d} \\
0 &\eqD& P_b^d \circ  [\Lambda^{d}]_D \;,
\end{eqnarray}
the \emph{relaxed problem of level $d$ and order $D$}. I.e., it is assumed that up to order $d$ the original problem is solved but only block diagonalized for the higher order coefficients.

To give an illustrative example consider this relaxed problem of order $3$ and level $2$. At this point of the algorithm, one has potentially obtained a matrix polynomial $[\Lambda]_3 = \sum_{d=0} \Lambda_d T^d$ with coefficients of the form
\begin{eqnarray*}
\Lambda_0 =
\begin{pmatrix}
1\\
&1 \\
&&1\\
&&&2\\
&&&&2\\
&&&&&3
\end{pmatrix}, \quad
\Lambda_1 = 
\begin{pmatrix}
2\\
&2 \\
&&3\\
&&&2\\
&&&&2\\
&&&&&2
\end{pmatrix}, \quad
\Lambda_2 = 
\begin{pmatrix}
1 & 3\\
3&5 \\
&&7\\
&&&1&2\\
&&&2&3\\
&&&&&7
\end{pmatrix} \;.
\end{eqnarray*}
I.e., $\Lambda_0$ and $\Lambda_1$ are already diagonal. Since there are two eigenvalues with multiplicity $m^2([\lambda]_3) = 2$ it follows that $\Lambda_2$ is only block diagonal.
Note that the eigenvalues are not globally sorted by value in the higher coefficients but only in the subblocks defined by the lower order coefficients. In this example, the repeated eigenvalues in the first block split at the lift from $d=0 $ to $d=1$. The blocks are defined by $b^1 = [1,4,6,7]$ and $b^2=[1,3,4,6,7]$. The blocks in $\Lambda_2$ are defined by $b^2$.

The function that solves the relaxed problem of order $D$ and level $d$ is denoted
\begin{eqnarray}
[\Lambda^d]_D, [Q^d]_D &=& \eigh_d( [A]_D) \;.
\end{eqnarray}
The idea is to implement an algorithm that successively increases $d$ by one.  For convenience we define $[Q^0]_D := \Id$ and $[\Lambda^0]_D := [A]_D$.

\begin{theorem}\label{thm:eigh}
Let $[A]_D$ be given. Then the solution of
\begin{eqnarray}
[ Q^{d+1}]_D, [\Lambda^{d+1}]_D \eqD \eigh_{d+1}([A]_D) 
\end{eqnarray}
can be computed from the solution
$ [ Q^d]_D, [\Lambda^d]_D \eqD \eigh_d([A]_D) $
by computing
\begin{eqnarray}
[\hat \Lambda_{s,s} ]_{D-d}, [\hat Q_{s,s} ]_{D-d} &\stackrel{D-d}{=}& \eigh_1 ( [\Lambda^d_{s,s} ]_{d:}) \label{eqn:eigh_1} \;,
\end{eqnarray}
where $s = {b^d_{\nb}:b^d_{\nb + 1}-1}$ are slice indices and $\nb = 1,\dots, \Nb^d$. All other elements of $[\hat Q]_{D-d}$ and $[\hat \Lambda]_{D-d}$ are zero. I.e., $[\hat Q]_{D-d}$ and $[\hat \Lambda]_{D-d}$ are block diagonal.
It holds that
\begin{eqnarray}
\; [\Lambda^{d+1} ]_D &\eqD& [\Lambda^d]_{d} + [\hat \Lambda]_{D-d} T^d \\
\; [Q^{d+1}]_D &\eqD& [Q^d]_D [Q]_D \;,
\end{eqnarray}
where $[Q]_D = [\hat Q]_{D-d} + [\Delta Q]_d T^{D-d}$ for some $[\Delta Q]_{D-d}$ that satisfies
\begin{eqnarray}
0 \eqD [Q]^T_D [Q]_D - \Id \label{eqn:q_lift}\;.
\end{eqnarray}

\end{theorem}
\begin{proof}
We need to show that $[\Lambda^{d+1}]_D$, $[Q^{d+1}]_D$ is a solution to the relaxed equations of level $d+1$ and order $D$.
From the definition of $\eigh_1$ it follows that $ 0 = (P_L + P_R) \circ [\Lambda^{d+1}]_{d+1}$ and  $0 = P_b^{d+1} \circ [\Lambda^{d+1}]_D$ is satisfied. We also know that $ 0 \eqD  [Q^{d+1}]_D^T   [Q^{d+1}]_D - \Id \eqD [Q]_D^T  [Q^{d}]_D^T  [Q^d]_D [Q]_D - \Id$ is fulfilled because   $ 0 \eqD  [Q^{d}]_D^T  [Q^{d}]_D - \Id$ and $ 0 \eqD [Q]_D^T  [Q]_D^{T} - \Id$. Hence, it only remains to show that the third defining equation is satisfied which is shown by the following straight-forward calculation:
\begin{eqnarray*}
0 &\eqD& [ Q]_D^T  [Q^{d}]_D^T [A]_D   [Q^{d}]_D [ Q]_D -  [\Lambda^{d+1}]_D \\
&\eqD& [ Q]_D^T  [\Lambda^d]_D [ Q]_D -  [\Lambda^{d+1}]_D \\
&\eqD& [ Q]_D^T ( [\Lambda^d]_{d} +  [\Lambda^d]_{d:} T^d ) [ Q]_D -  [\Lambda^{d}]_{d} -  [\hat \Lambda]_{D-d} T^d \\
&\eqD& [ Q]_D^T  [\Lambda^d]_{d} [ Q]_D + [ Q]_D^T  [\Lambda^d]_{d:} [ Q]_D T^d -  [\Lambda^{d}]_{d} -  [\hat \Lambda]_{D-d} T^d \\
&\eqD&[\Lambda^d]_{d} [ Q]_D^T [ Q]_D
 + [\hat Q]_{D-d}^T  [\Lambda^d]_{d:} [\hat Q]_{D-d} T^d -  [\Lambda^{d}]_{d} -  [\hat \Lambda]_{D-d} T^d \\
&\eqD&   [\hat Q]_{D-d}^T  [\Lambda^d]_{d:} [\hat Q]_{D-d} T^d -  [\hat \Lambda]_{D-d} T^d \\
&\stackrel{D-d}{=}&   [\hat Q]_{D-d}^T  [\Lambda^d]_{d:} [\hat Q]_{D-d}  -  [\hat \Lambda]_{D-d} \;.
\end{eqnarray*}
In the fifth line it has been used that the diagonalization has only to be performed for block diagonal matrices. If the eigenvalues are already distinct there is nothing to diagonalize and the step can be skipped. It also means that one may interchange $[\Lambda^d]_{d}$ with $[Q]_D$.
\end{proof}

The following proposition gives us the means to diagonalize a matrix in the zeroth degree and block diagonalize w.r.t. the blocks defined by the repeated eigenvalues.
I.e., it gives the justification that the solution of \Eqn{eqn:eigh_1} can be found. In the case of distinct eigenvalues the application of this algorithm already solves the original problem.
\begin{proposition}\label{prop:order_one_relaxed_problem}
Let $[A]_{D+E} = [A]_D + [\Delta A]_E T^D$ be given and $[\Lambda^d]_D$, $[Q^d]_D$ be a solution of the relaxed problem of level $d=1$ and order $D$. Then it exist $[\Delta \Lambda^d]_E$ and $[\Delta Q^d]_E$ such that $[\Lambda^d]_{D+E} = [ \Lambda^d]_D + [\Delta \Lambda^d]_E T^D$ and 
$[\Delta Q^d]_{D+E} = [\Delta Q^d]_D + [\Delta Q^d]_E T^D$ are a solution of the relaxed problem of level $d=1$ and order $D+E$. A closed form solution is
\begin{eqnarray}
\; [\Delta \Lambda^d]_E & \eqE & \bar P_b^d \circ [K]_E \\
\; [\Delta Q^d]_E & \eqE & [Q^d]_E \left( [\Delta G]_E + P_b^d \circ \left( [K]_E/[E]_E \right) \right)
\end{eqnarray}
where $[\Delta F]_E T^D \eqDE [Q^d]^T_D [A]_D [Q^d]_D - [\Lambda^d]_D$ and $[\Delta G]_E T^D \eqDE -\frac{1}{2} \left( [Q^d]^T_D [Q^d]_D - \Id \right)$, $[K]_E \eqE  [\Delta F]_E + ( [\Lambda]_E [\Delta G]_E + [\Delta G]_E [\Lambda]_E ) + [Q^d]^T_E [\Delta A]_E [Q^d]_E $ and $[E_{ij}]_E \eqE [\Lambda^d_{jj}]_E - [\Lambda^d_{ii}]_E$. The expression $[K]_E/[E]_E$ denotes an element-wise division. $P_b^d$ is a matrix with only ones on the diagonal blocks defined by the multiplicity of eigenvalues in $\Lambda_0$. $\bar P_b^d$ is defined s.t. $\bar P_b^d + P_b^d$ is a matrix full of ones. One can see here that if the eigenvalues are distinct, then $\bar P_b^d$ is the identity matrix $\Id$.
\end{proposition}
\begin{proof}
We set $Q^d \equiv Q$ etc. for notational simplicity.
Applying Newton-Hensel lifting to the defining equations yields
\begin{eqnarray}
0 & \eqDE & ([Q]_D + [\Delta Q]_E T^D)^T ([Q]_D + [\Delta Q]_E T^D) - \Id \nonumber \\
 & \eqE & - 2 [\Delta G]_E + [\Delta Q]_E^T [Q]_E + [Q]_E^T [\Delta Q]_E \nonumber\\
 & \eqE & -2 [\Delta G]_E + 2 [S]_E \;, \nonumber\\
0 & \eqDE & ([Q]_D + [\Delta Q]_E T^D)^T ( [A]_D + [\Delta A]_E T^D) ([Q]_D + [\Delta Q]_E T^D) - ([\Lambda]_D + [\Delta \Lambda]_E T^D) \nonumber \\
& \eqE & [\Delta F]_E + [Q]_E^T [\Delta A]_E [Q]_E +  [\Delta Q]_E^T [Q]_E [\Lambda]_E + [\Lambda]_E [Q]_E^T [\Delta Q]_E - [\Delta \Lambda]_E \nonumber \\
& \eqE & [K]_E + [X]_E [\Lambda]_E - [\Lambda]_E [X]_E - [\Delta \Lambda]_E \nonumber \\
& \eqE & [K]_E + [E]_E \circ [X]_E - [\Delta \Lambda]_E \;.
\end{eqnarray}
Thus $[\Delta \Lambda]_E \eqE \bar P_b^d \circ [K]_E$ and $[X]_E^T \eqE P_b^d \circ ( [K]_E/[E]_E)$.
Above, $[\Delta Q]_E^T [Q]_E \eqE [S]_E + [X]_E$, $[S]_E$ symmetric and $[X]_E$ antisymmetric (Lemma \ref{lemma:symmetric_antisymmetric}) has been used.
\end{proof}

It remains to show that \Eqn{eqn:q_lift} can be satisfied.
\begin{lemma}\label{lemma:qlift}
Let $[Q]_D$ be given and it satisfies the defining equation $ 0 \eqD [Q]_D^T [Q]_D - \Id$. Then the solution can be lifted to $D+E$ with $E\leq D$. I.e.,  it is possible to find $[Q]_{D+E} := [Q]_D + [\Delta Q]_E T^D$ s.t. $ 0 \eqDE [Q]_{D+E}^T [Q]_{D+E} - \Id$. A closed form solution for $[\Delta Q]_E $ is given by
\begin{eqnarray}
\; [\Delta Q]_E & \eqE &  [Q]_E [S]_E  \;,
\end{eqnarray}
where $[S]_E T^D \stackrel{D+E}{=} - \frac{1}{2} \left( [Q]_D^T [Q]_D - \Id \right)$.
\end{lemma}
\begin{proof}
In Appendix \ref{appendix:qlift}.
\end{proof}

\subsection{Pullback} 
The eigenvalue decomposition is non-differentiable at points where eigenvalues are repeated and hence the defining equations do not define a \emph{well behaved implicit mapping} as described by Christianson \cite{Christianson1998RAa}.
However, the eigenvalue decomposition is typically used within a global context where the non-uniqueness and non-differentiability can be worked around. Here, we give only the pullback algorithm that is correct for unique eigenvalues.

\begin{proposition}[Pullback of the Symmetric Eigenvalue Decomposition with Distinct Eigenvalues:]\label{prop:eigh_pullback}
Given $\textred{A}, \textred{Q}, \textred{\Lambda}, \textred{\bar Q}, \textred{\bar \Lambda}$, where all eigenvalues are distinct, one can compute $\textgreen{\bar A}$ by
\begin{eqnarray}
\; H_{ij}&=& (\textred{\lambda_j} - \textred{\lambda_i})^{-1} \quad \mbox{if} \quad  i \neq j , \quad 0 \quad \mbox{ else} \\
\textgreen{\bar A} &=& \textred{Q} \left( \textred{\bar \Lambda} + H \circ (\textred{Q^T \bar Q})\right) \textred{Q^T}
\end{eqnarray}
\end{proposition}
\begin{proof}
In Appendix \ref{appendix:eigh_pullback}.
\end{proof}

\subsection{Explicit algorithms}

\begin{algorithm2e}
\SetKwInOut{Input}{input}\SetKwInOut{Output}{output}
\SetLine 
\Input{$[Q]_{d} = [Q_0, \dots, Q_{d-1}]$ with $0 \stackrel{d}{=} [Q]_d^T [Q]_d - \Id$ }
\Input{$D \in \N$}
\Output{$[Q]_{D} = [Q_0, \dots, Q_{D-1}]$, where  $0 \stackrel{D}{=} [Q]_D^T [Q]_D - \Id$}
\BlankLine
\BlankLine
\For{$k = d$ \KwTo $D-1$}{
$Q_{k}  =-  \frac{1}{2}Q_0  \sum_{i =1}^{k-1} Q_i^T Q_{k-i}$
}
\caption{ This algorithm computes $[Q]_D = \mbox{qlift}([Q]_d,D)$ as described in Proposition \ref{lemma:qlift} using sequential Hensel-lifting ($E=1$).}\label{algo:qlift}
\end{algorithm2e}
\begin{algorithm2e}
\SetKwInOut{Input}{input}\SetKwInOut{Output}{output}

\SetLine 
\Input{$[A]_{D} = [A_0, \dots, A_{D-1}]$, where $A_d \in \R^{N \times N}$ symmetric positive definite, $d=0,\dots,D-1$ }
\Output{$[\Lambda]_{D} = [\Lambda_0, \dots, \Lambda_{D-1}]$, where $\Lambda_0 \in \R^{N \times N}$ diagonal and $\Lambda_d \in \R^{N \times N}$ block diagonal  $d=1,\dots,D-1$. }
\Output{$[Q]_{D} = [Q_0, \dots, Q_{D-1}]$ orthogonal, where $Q_d \in \R^{N \times N}$ }
\Output{$b \in \N^{N_b+1}$, array of integers defining the blocks. The integer $N_b$ is the number of blocks. Each block has the size of the multiplicity of an eigenvalue $\lambda_{n_b}$ of $\Lambda_0$ s.t. for $s = b_{n_b}:b_{n_b+1}-1$ one has $(Q_{0;:,s})^T A_0 Q_{0;:,s} = \lambda_{n_b} \Id$.}

\BlankLine
$ \Lambda_0, Q_0 = \eigh (A_0)$\\
\BlankLine
compute $b \in \R^{N_b + 1}$ \\
$E_{ij} = \Lambda_{0;jj} - \Lambda_{0;ii}$ \\
$H = P_b \circ (1/E)$ \\
\BlankLine
\For{$d = 1$ \KwTo $D-1$}{
$ \Delta F = \sum_{|i|=d} Q_{i_1}^T A_{i_2} Q_{i_3} $ \\
$ S = - \frac{1}{2} \sum_{k=1}^{d-1} Q^{T}_{d-k} Q_k $ \\
$ K = \Delta F +  Q_0^T A_d  Q_0 + S  \Lambda_0 +  \Lambda_0 S $\\
$ Q_d = Q_0 ( S + H \circ K)$\\
$ \Lambda_d = \bar P_b \circ K$
}
\caption{\label{algo:blockdiag} This algorithm computes $[\Lambda]_D, [Q]_D, b = \eigh_1([A]_D)$ as specified by \ref{prop:order_one_relaxed_problem} using sequential Hensel-lifting ($E=1$). I.e., the zeroth coefficient is diagonalized and the higher order coefficients are block diagonalized. The symbol $i \in \N_0^3$ denotes a multi-index, i.e., the summation $\sum_{|i|=d}$ goes over all possible $i$ such that $|i| \equiv \sum_{k=1}^3 i_k = d$. }
\end{algorithm2e}
\begin{algorithm2e}
\SetKwInOut{Input}{input}\SetKwInOut{Output}{output}
\SetLine 
\Input{$[A]_{D} = [A_0, \dots, A_{D-1}]$ symmetric with $A_d \in \R^{N \times N}$ }
\Output{$[\Lambda]_{D} = [\Lambda_0, \dots, \Lambda_{D-1}]$, where $\Lambda_d \in \R^{N \times N}$ diagonal for $d=0,\dots,D-1$. }
\Output{$[Q]_{D} = [Q_0, \dots, Q_{D-1}]$ orthogonal, where $Q_d \in \R^{N \times N}$ }
\BlankLine
$[\Lambda^0]_D = [A]_D$, $[Q^0]_D = \Id$ and $b^0 = [1,N+1]$\\
\For{$d = 0$ \KwTo $D-1$}{
\For{$n_b = 1$ \KwTo $N_b^d$}{
$s = b^d_{n_b} : b^d_{n_b+1} -1$ \quad (slice index)\\
$[\hat \Lambda_{s,s}]_{D-d}, [\hat Q_{s,s}]_{D-d}, b^{d+1} = \eigh_1( [\Lambda^d_{s,s}]_{d:})$\\
$[Q_{s,s}]_D = \mbox{qlift}([\hat Q_{s,s}]_{D-d}, D)$
}
$[\Lambda^{d+1}]_D = [\Lambda^d]_d + [\hat \Lambda]_{D-d} T^d $\\
$[Q^{d+1}]_D = [Q^d]_D [Q]_D$
}
\caption{This algorithm computes $[\Lambda]_D, [Q]_D = \eigh([A]_D)$ as described in Theorem~\ref{thm:eigh}. The algorithm uses internally Algorithm \ref{algo:blockdiag} and \ref{algo:qlift}.}\label{algo:eigh} 
\end{algorithm2e}

\pagebreak[6]

\section{Numerical tests and examples} \label{sec:numerical_results}
\subsection{Taylor polynomial arithmetic on real symmetric eigenvalue problem}
As an example to test the validity of the pushforward in UTP arithmetic we consider the following system \cite{andrew:98}:
\begin{eqnarray*}
Q(t) &=& \frac{1}{\sqrt{3}} 
\begin{pmatrix}
\cos(x(t)) & 1 & \sin(x(t)) & -1 \\
-\sin(x(t)) & -1 & \cos(x(t)) & -1 \\
1 & - \sin(x(t)) & 1 & \cos(x(t)) \\
-1 & \cos(x(t)) & 1 & \sin(x(t)) \\
\end{pmatrix} \\
\Lambda(t) &=& \diag (x^2 - x + \frac{1}{2},  4 x^2 - 3 x, \delta ( -\frac{1}{2} x^3 + 2 x^2 - \frac{3}{2} x + 1) + ( x^3 + x^2 - 1), 3x - 1  ) \;,
\end{eqnarray*}
where $x \equiv x(t) := 1 + t$. The constant $\delta$ is some predefined constant. In Taylor arithmetic one obtains
\begin{eqnarray*}
\Lambda_0 &=& \diag (1/2, 1, 1+ \delta ,2 )\\
\Lambda_1 &=& \diag (1, 5, 5+\delta, 3) \\
\Lambda_2 &=& \diag (2, 8, 8 + \delta , 0) \\
\Lambda_3 &=& \diag (0,0,6 - 3 \delta,  0) \\
\Lambda_d &=& \diag (0,0, 0, 0), \quad \forall d \geq 4 \;.
\end{eqnarray*}
One can see that in the case $\delta = 0$ one obtains one repeated eigenvalue that splits at $d=3$.
We apply Algorithm \ref{algo:eigh} to reconstruct $[\Lambda]_D$. The reconstructed values are denoted $[\tilde \Lambda]_D$.
The numerical results are shown in \Fig{fig:andrew_test_example}.

\begin{figure}[!ht]
\centering
\includegraphics[width=0.495\textwidth]{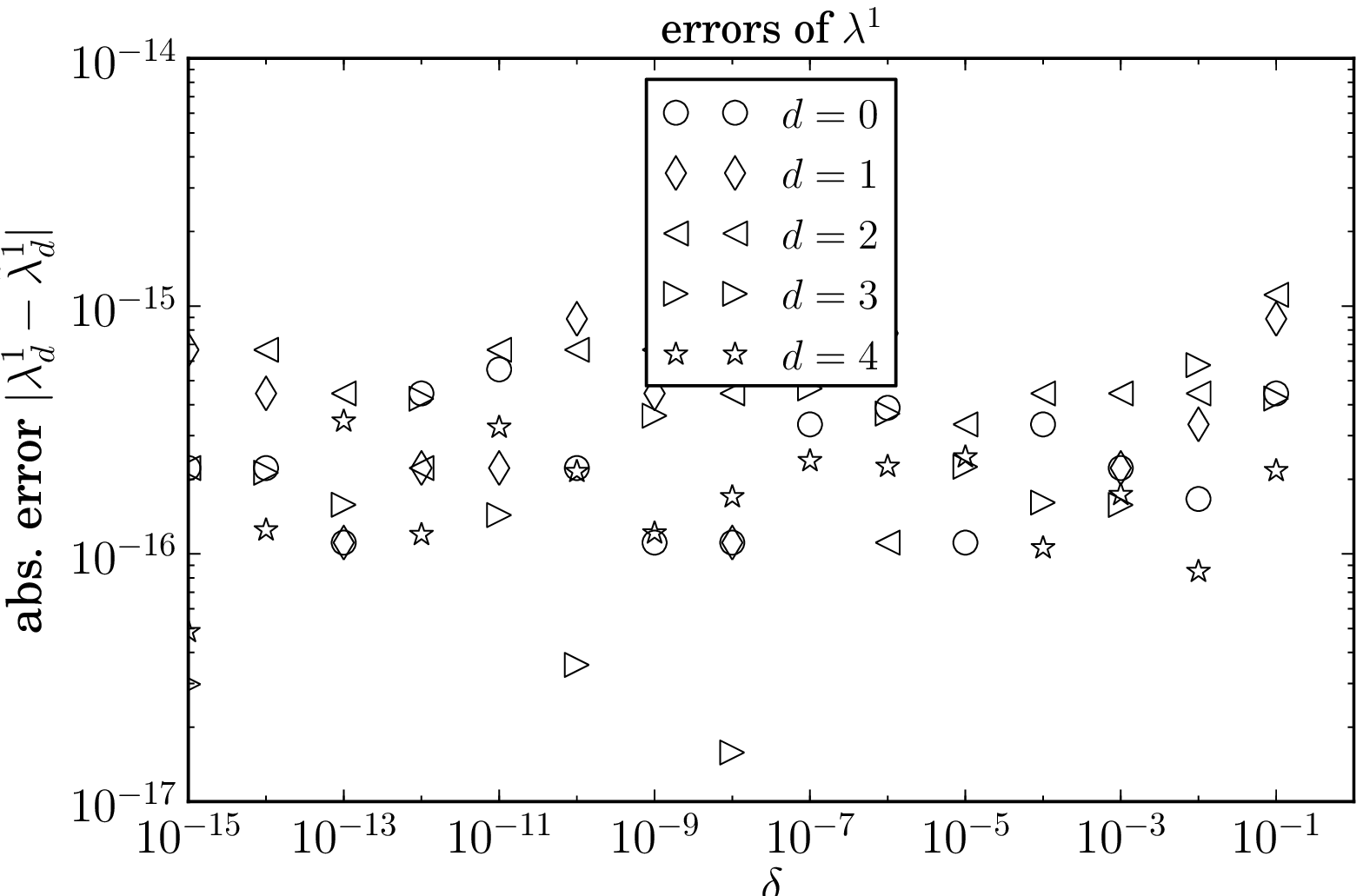}
\includegraphics[width=0.495\textwidth]{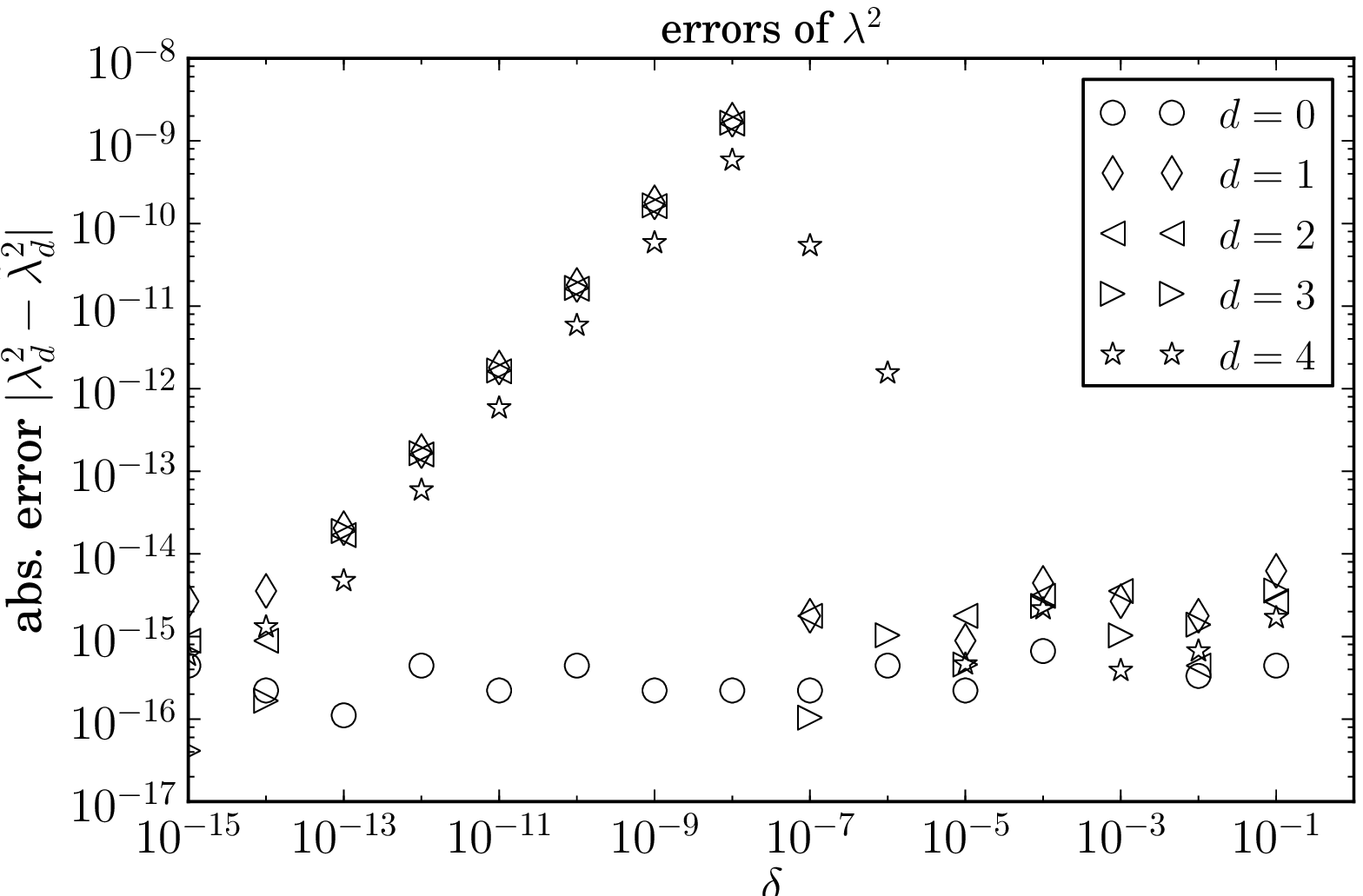}

\caption{\label{fig:andrew_test_example} One left side one can see that error between the true $\lambda_1$ and reconstructed $\tilde \lambda_1$ is close to machine precision. On the right side one can see that the absolute error $| \tilde \lambda_2 - \lambda_2|$ has a jump at $\delta \approx 10^{-7}$. This is due to the fact that that the algorithm treats eigenvalues $|\lambda_i - \lambda_j| < 10^{-7}$ as repeated eigenvalues. One can see that when $\delta$ approaches $10^{-16}$ the error gets smaller. The eigenvalue $\lambda_4$ shows the same qualitative behavior as $\lambda_1$ and $\lambda_3$ the same as $\lambda_2$.}
\end{figure}

\subsection{Covariance matrix computation}
To test the validity of the covariance matrix computation of \Eqn{eqn:covariance_matrix} we first check that the first directional derivatives of the covariance matrix $C$ w.r.t. $J_1$ and $J_2$ coincide with the results from the complex-step derivative approximation, abbreviated here for convenience as CSDA. The CSDA computes directional derivatives of a real-valued function $y = f(x)$ as $\dot y \approx \frac{\Im(f(x+i\epsilon \dot x))}{\epsilon} = \frac{f(x + i \epsilon \dot x) - f(x - i \epsilon \dot x)}{2i\epsilon} $, i.e., $\Im$ extracts the imaginary part and $i \equiv \sqrt{-1}$. The number $\epsilon \in \R$ can be made very small. For a detailed discussion that also shows the relation to AD see Martins et al. \cite{Martins:2001, Martins:2003}. Having verified the first order derivatives by UTP arithmetic we can check if the UTP arithmetic on \Eqn{eqn:cov_qr} yields the same result. Unfortunately, it is not possible to use the CSDA in a straight-forward fashion since for complex matrices the QR decomposition does not yield an orthogonal but a unitary $Q$. For reproducibility we define $J_1$ and $J_2$ rather arbitrarily as
\begin{eqnarray*}
J_1(x) &=& 
\begin{pmatrix}
\sin(x_1) x_2 & \cos(x_2) \\
\exp(x_1) & x_1 x_2 \\
x_1 \log(x_2) & \log( 1 + \exp(\cos(x_1))) \\
x_2 + x_1 & x_1 (x_2 + \cos(x_1) \\
\end{pmatrix} \;, \quad \; J_2(x)^T =
\begin{pmatrix}
x_1 \log(x_2 + 3 \sin(x_1x_2))\\ x_2 \exp( \sin(x_1) + \cos(x_1 x_2 ))
\end{pmatrix} \;.
\end{eqnarray*}
The numerical results are shown in Figure \ref{fig:cov}. Note that $x_1$ and $x_2$ are here elements of the vector $x$ and not coefficients.

\begin{figure}[!ht]
\centering
\includegraphics[width=0.495\textwidth]{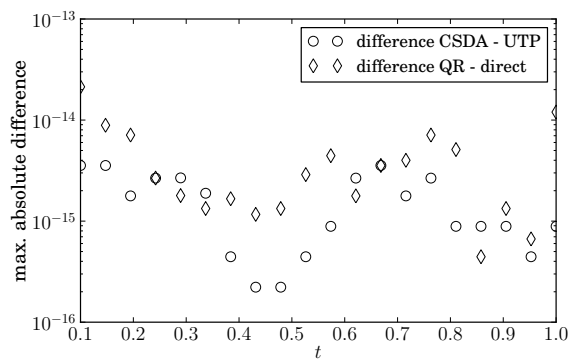}

\caption{\label{fig:cov} This plot shows the maximum absolute differences of the directional derivatives at  $ x = t(3,1)^T$, where $t \in [0.1, 1]$ in direction $\dot x = (5,7)^T$. The circles show the difference between the CSDA solution and the first order UTP solution using \Eqn{eqn:covariance_matrix}. The diamonds show the difference between the UTP solution of \Eqn{eqn:covariance_matrix} and \Eqn{eqn:cov_qr}. One can see that the difference is close to machine precision of IEEE 754 float64, which is approximately $10^{-16}$.}
\end{figure}

\section{Summary and outlook}
We have shown how computer codes containing real symmetric eigenvalue decompositions and QR decompositions can be evaluated in univariate Taylor polynomial arithmetic. Furthermore, the reverse mode of AD has been treated. Explicit algorithms have been presented that can be used in combination with existing AD software, e.g. general purpose AD tools like ADOL-C \cite{Griewank1999ACA} or CppAD \cite{Bell:2010} but also differentiated DAE solvers like SolvIND \cite{SOLVIND-hp}.
Numerical tests have been used to check the algorithms.

Other algorithms that contain the the QR decomposition and the real symmetric eigenvalue decompostion can be differentiated using the shown algorithms. In particular, we think of the differentiation of the SVD or the computation of pseudoinverses. We believe that these algorithms, in modified form, may also be valuable for the eigenvalue optimization problem where eigenvalues are repeated in the solution point. 

\section*{Acknowledgements}
The authors wish to thank Bruce Christianson for his comments that helped to gain a deeper understanding of the matter and also greatly helped to improve the readability of this work.

This research was partially supported by the Bundesministerium f\"ur Bildung und Forschung (BMBF) within the project NOVOEXP (Numerische Optimierungsverfahren f\"ur die Parametersch\"atzung und den Entwurf optimaler Experimente unter Ber\"ucksichtigung von Unsicherheiten f\"ur die Modellvalidierung verfahrenstechnischer Prozesse der Chemie und Biotechnologie) (03GRPAL3), Humboldt Universit\"at zu Berlin.

\appendix
\addcontentsline{toc}{section}{Appendices}
\section{Additional proofs}
\subsection{Proofs of QR decomposition} 
\subsubsection{Proof of  Proposition \ref{prop:qr}}\label{appendix:utpm_qr}
\begin{proof}
We look at the first defining equation and try to separate the known from the unknown quantities:
\begin{eqnarray}
0 &\eqDE& [Q]_{D+E} [R]_{D+E} - [A]_{D+E} \nonumber \\
&\eqDE& ([Q]_D + [\Delta Q]_E T^D) ( [R]_D + [\Delta R]_E T^D) - [A]_{D+E} \nonumber \\
&\eqDE& [Q]_D [R]_D - [A]_{D+E}  + ( [\Delta Q]_E [R]_D + [Q]_D [\Delta R]_E ) T^D \nonumber \\
&\eqDE& -[\Delta F]_E T^D + ( [\Delta Q]_E [R]_D + [Q]_D [\Delta R]_E ) T^D \nonumber \\
&\eqE& -[\Delta F]_E +  [\Delta Q]_E [R]_E + [Q]_E [\Delta R]_E \label{eqn:qrlift1}\;.
\end{eqnarray}
Similarly for the second defining equation
\begin{eqnarray*}
0 & \eqDE& [Q]^T_{D+E} [Q]_{D+E} - \Id \\
  & \eqDE& [Q]^T_{D} [Q]_{D} - \Id + ( [Q]^T_{D} [\Delta Q]_{E} +  [\Delta Q]^T_{E} [Q]_{D}) T^D \\
\Rightarrow 0  & \eqE & -[\Delta G]_E +  [Q]^T_{E} [\Delta Q]_{E} +  [\Delta Q]^T_{E} [Q]_{E} \\
 &\eqE&  -[\Delta G]_E + [S]_E + [X]_E + [S]_E - [X]_E \\
\Rightarrow  \quad S &=&  \frac{1}{2} [\Delta G]_E \;,
\end{eqnarray*}
where $[S]_E + [X]_E = [Q]_E^T [\Delta Q]_E$ and it has been used that every matrix can be written as the sum of a symmetric and an antisymmetric matrix.
Now multiply (\ref{eqn:qrlift1}) by $[Q]^T_E$ from the left to obtain
\begin{eqnarray}
0 & \eqE & - [Q]^T_E  [\Delta F]_E + [Q]^T_E [\Delta Q]_E [R]_E + [\Delta R]_E \label{eqn:qrlift2}\\
 & \eqE & - [Q]^T_E  [\Delta F]_E + ([S]_E + [X]_E) [R]_E + [\Delta R]_E \label{eqn:qrlift21} \\
 & \eqE & - [Q]^T_E  [\Delta F]_E + [S]_E [R]_E + [X]_E [R]_E + [\Delta R]_E \nonumber \;.
\end{eqnarray}
Multiplication of $[R_{:N,:}]_E^{-1}$ from the right yields
\begin{eqnarray*}
0 & \eqE &  - [Q]^T_E  [\Delta F]_E  [R_{:N,:}]_E^{-1} + [S]_E [R]_E [R_{:N,:}]_E^{-1}+ [X]_E [R]_E [R_{:N,:}]_E^{-1} + [\Delta R]_E [R_{:N,:}]_E^{-1} \\
& \eqE & - [Q]^T_E  [\Delta F]_E  [R_{:N,:}]_E^{-1} + [S_{:,:N}]_E + [X_{:,:N}]_E  + [\Delta R]_E [R_{:N,:}]_E \\
\Rightarrow   P_L \circ ( [X_{:,:N}]_E) &\eqE& P_L \circ \left( [Q]_E^T [\Delta F]_E [R_{:N,:}]_E^{-1} - [S_{:,:N}]_E \right) \;.
\end{eqnarray*}
The coefficients of $X_{:,N+1:}$ are not specified and can for instance be set to zero. Since $X$ is antisymmetric it is already defined by the above equation. Since $[S]_E + [X]_E \eqE [Q]_E^T [\Delta Q]_E$ one can obtain $[\Delta Q]_E$ as
\begin{eqnarray*}
[\Delta Q]_E &=& [Q]_E ( [S]_E + [X]_E) 
\end{eqnarray*}
because for quadratic $Q$ one has the identity $Q Q^T = \Id$.
\end{proof}

\subsubsection{Proof of Proposition \ref{prop:pullback_qr}}\label{appendix:reverse_qr}
\begin{proof}
We differentiate the implicit system
\begin{eqnarray*}
0 &=& A - Q R \\
0 &=& Q^T Q - \Id \\
0 &=& P_L \circ R
\end{eqnarray*}
and obtain
\begin{eqnarray*}
0 &=& \dd A - \dd Q R - Q \dd R  \quad (*)\\
0 &=& \dd Q^T Q + Q^T \dd Q \quad (**)\;.
\end{eqnarray*}
We define the antisymmetric ``matrix'' $ X := Q^T \dd Q$. Multiplication of Eqn. (*) from the left with $Q^T$ yields
\begin{eqnarray*}
0 &=& Q^T \dd A - XR - \dd R \\
\mbox{hence} \quad \dd R &=& Q^T \dd A - X  R \;.
\end{eqnarray*}
The multipication of this last equation from the right with the Moore-Penrose pseudoinverse $R^+=(R_{:N,:}{}^{-1},0)$ yields the equivalent equation
\begin{eqnarray*}
0 &=& Q^T \dd A R^+ - X R R^+ - \dd R R^+ \\
\mbox{ and thus } \; P_L \circ X &=& P_L \circ( Q^T \dd A R^+ ) \;,
\end{eqnarray*}
where we have chosen arbitrarily that $X_{:,N+1:} = 0$. Since $X$ is antisymmetric we have
\begin{eqnarray*}
X &=& (P_L \circ X) - (P_L \circ X)^T \;.
\end{eqnarray*}
We can use these results to compute the pullback:
\begin{align*}
\tr( \bar R^T \dd R) + \tr( \bar Q^T \dd Q) 
&= \tr ( Q \bar R \dd A^T) - \tr( R \bar R^T X) + \tr( \bar Q^T Q Q^T \dd Q) \\
&= \tr ( Q \bar R \dd A^T) + \tr ( \underbrace{( \bar Q^T Q - R \bar R^T)}_{=:K} X) \\
&= \tr ( Q \bar R \dd A^T) + \tr ( ( K - K^T) ( P_L \circ X)) \\
&= \tr ( Q \bar R \dd A^T) + \tr ( R^{+T} \dd A^T Q ( P_L \circ ( K^T - K))) \\
&= \tr ( Q [ \bar R + \{ P_L \circ ( Q^T \bar Q - \bar Q^T Q + R \bar R^T - \bar R R^T) \} R^{+T}] \dd A^T)\\
&= \tr ( \bar A \dd A^T) \;.
\end{align*}
In the above derivation we have used  Lemmas \ref{lemma_anti_symmetric}, \ref{lemma_transpose_projector} and \ref{lemma_abc}.

\end{proof}

\subsubsection{Proof of Lemma \ref{lemma:qlift}} \label{appendix:qlift}
\begin{proof}
\begin{eqnarray*}
0 &\eqDE& ([Q]_D + [\Delta Q]_E T^D)^T ([Q]_D + [\Delta Q]_E T^D) - \Id\\
&\eqDE& ([Q]_D^T [Q]_D - \Id) + ([Q]_D^T [\Delta Q]_E + [\Delta Q]_E^T [Q]_D) T^D \\
&\eqE& [\Delta G]_E + [Q]_E^T [\Delta Q]_E + [\Delta Q]_E^T [Q]_E \\
&\eqE&  [\Delta G]_E + 2 [S]_E \\
\; [\Delta Q]_E &\eqE& - \frac{1}{2} [Q]_E [\Delta G]_E\;,
\end{eqnarray*}
where $ [\Delta Q]_E^T [Q]_E = [S]_E + [X]_E$, $[S]_E$ symmetric and $[X]_E$ antisymmetric and $[\Delta G_E]_E T^D \stackrel{D+E}{=} \left( Q^T Q - \Id \right)$ . Since no condition defines constraints on $[X]_E$ it has been set to zero.
\end{proof}

\subsubsection{Proof of Proposition \ref{prop:eigh_pullback}} \label{appendix:eigh_pullback}
\begin{proof}
We want to compute $ \tr( \bar A^T \dd A) = \tr( \bar \Lambda^T \dd \Lambda) + \tr( \bar Q^T \dd Q)$. We differentiate the implicit system
\begin{eqnarray*}
0 &=& Q^T A Q - \Lambda \\
0 &=& Q^T Q - \Id\\
0 &=& (P_L + P_R) \circ \Lambda
\end{eqnarray*}
and obtain
\begin{eqnarray*}
\dd \Lambda &=& Q^T \dd A Q + \dd Q^T A Q + Q^T A \dd Q \\
	&=&Q^T \dd A Q + \dd Q^T Q\Lambda + \Lambda Q^T\dd Q \\
0 &=& \dd Q^T Q + Q^T \dd Q \;.
\end{eqnarray*}
A straight forward calculation shows:
\begin{eqnarray*}
\tr ( \bar \Lambda^T \dd \Lambda) 
&=& \tr ( Q \bar \Lambda Q^T \dd A) + \tr ( \Lambda \bar \Lambda \dd Q^T Q ) + \tr( \bar \Lambda \Lambda Q^T \dd Q) \\
&=& \tr ( Q \bar \Lambda Q^T \dd A) \;, \\
\tr( \bar Q^T \dd Q) 
&=& \tr (\bar Q^T Q (H \circ (Q^T \dd A Q))) \\
&=& \tr ( Q ( H^T \circ (\bar Q^T Q)) Q^T \dd A) \;, \\
\tr ( \bar A^T \dd  A) &=& \tr \left( ( Q (\bar \Lambda + H^T \circ ( \bar Q^T Q)) Q^T )\dd A \right)
\end{eqnarray*}
where we have used
\begin{eqnarray*}
\dd \Lambda 
  &=&Q^T \dd A Q - (Q^T\dd Q)^T\Lambda + \Lambda Q^T\dd Q \\
  &=&Q^T \dd A Q - K \circ (Q^T \dd Q)\\
\implies
  Q^T \dd Q &=& H \circ ( Q^T \dd A Q - \dd \Lambda) \\
  &=& H \circ ( Q^T \dd A Q)
\end{eqnarray*}
where we have defined $K_{ij} := \Lambda_{jj} - \Lambda_{ii}$ and $H_{ij} = (K_{ij})^{-1}$ for $i \neq j$ and $H_{ij} = 0$ otherwise and used the property $X \Lambda - \Lambda X  = K \circ X$ for all $X \in \R^{N \times N}$ and diagonal $\Lambda \in \R^{N \times N}$.
\end{proof}

\subsection{Basic results used in the proofs}

\begin{lemma}\label{lemma_anti_symmetric}
Let $X \in \R^{N \times N}$ be an antisymmetric matrix, i.e., $X^T = - X$ and $P_L$ defined as above. We then can write
\begin{eqnarray}
X &=& P_L \circ X - (P_L \circ X)^T \;.
\end{eqnarray}
\end{lemma}
\begin{proof}
$
X = P_L \circ X + P_R \circ X = P_L \circ X + (P_L \circ X^T)^T = P_L \circ X - (P_L \circ X)^T
$
\end{proof}
\begin{lemma}\label{lemma_transpose_projector}
Let $A \in \R^{N \times N}$ and $P_L$ resp. $P_R$ defined as above. Then
\begin{eqnarray}
(P_L \circ A)^T &=& P_R \circ A^T \;.
\end{eqnarray}
\end{lemma}
\begin{proof}
$B_{ij} := (P_L \circ A)_{ij} = A_{ij} (i > j)$ and $B_{ij}^T = B_{ji} = A_{ji} (j>i) = A_{ij}^T P_R = P_R \circ A
$
\end{proof}

\begin{lemma}\label{lemma_abc}
Let $A,B,C \in \R^{M \times N}$. We then have
\begin{eqnarray}
\tr \left( A^T ( B \circ C)  \right) &=& \tr \left( C^T ( B \circ A) \right)
\end{eqnarray}
\end{lemma}
\begin{proof}
$\tr ( A^T (B \circ C)) = \sum_{i = 1}^N \sum_{j  = 1}^M A_{ij} B_{ij} C_{ij}= \tr ( C^T ( B \circ A ))$
\end{proof}

\begin{lemma} \label{lemma:low_trig_AB}
Let $A,B$ be lower triangular matrices. Then the following expression holds:
\begin{eqnarray}
P_D \circ ( A B) &=& (P_D \circ A) ( P_D \circ B) \;. 
\end{eqnarray}
\end{lemma}
\begin{proof}
$AB$ is also lower trinangular and thus 
$P_D \circ (AB)=diag(a_{ii}b_{ii})=diag(a_{ii})diag(b_{ii})=(P_D \circ A) (P_D \circ B)$
\end{proof}

\begin{lemma} \label{lemma:low_trig_AT}
The formula
\begin{eqnarray}
P_D \circ (A^T) &=& P_D \circ A
\end{eqnarray}
holds for all quadratic matrices $A$.
\end{lemma}
\begin{proof}
$(P_D \circ (A^T))_{ij} = \delta_{ij} A_{ji} = \delta_{ij} A_{ij} = (P_d \circ A)_{ij}$
\end{proof}

\begin{lemma} \label{lemma:low_trig_Ainv}
Let $A$ be a nonsingular quadratic lower triangular matrix. Then the formula
\begin{eqnarray}
P_D \circ (A^{-1}) &=& (P_D \circ A)^{-1}
\end{eqnarray}
holds.
\end{lemma}
\begin{proof}
Using Lemma \ref{lemma:low_trig_AB} we obtain
$(P_D \circ (A^{-1})) (P_D \circ A) = P_D \circ \Id = \Id$. Since the quadratic matrices form a group, the inverse is unique. Therefore, equality between $(P_D \circ (A^{-1})) = (P_D \circ (A))^{-1}$ must hold.
\end{proof}

\begin{lemma} \label{lemma:strict_low_trig_AB}
Let $A \in \R^{N \times N}$ be strictly lower triangular and $B \in \R^{N \times N}$ lower triangular. Then their product $C = A B$ is strictly lower triangular.
\end{lemma}
\begin{proof}
$C$ is lower triangular and the diagonal entries are $C_{ii}=A_{ii}B_{ii}=0$
since $B$ has a zero diagonal.
\end{proof}

\begin{corollary} \label{corollary:strict_low_trig_AB}
Let $A \in \R^{N \times N}$ be strictly lower triangular and $D \in \R^{N \times N}$ diagonal. Then their product $C = A D$
is strictly lower triagonal.
\end{corollary}

\begin{lemma} \label{lemma:symmetric_antisymmetric}
Every quadratic matrix $A$ can be written as the sum of a symmetric matrix $S = \frac{1}{2} ( A + A^T)$ and an antisymmetric matrix $X = \frac{1}{2} ( A - A^T)$, i.e.
\begin{eqnarray}
A &=& S + X 
\end{eqnarray}
\end{lemma}
\begin{proof}
$A = \frac{1}{2} ( A + A^T + A - A^T) = \frac{1}{2} ( A + A^T) + \frac{1}{2} ( A - A^T) = S + X$.
\end{proof}

\bibliographystyle{alpha}
\bibliography{refs}

\end{document}